\documentclass[11pt,a4paper]{article}

\usepackage{amsmath,amsthm,amssymb}
\usepackage{relsize, tikz}
\usepackage{enumitem}

\usetikzlibrary{matrix}
\usetikzlibrary{calc}
\tikzset{
    ncbar angle/.initial=90,
    ncbar/.style={
        to path=(\tikztostart)
        -- ($(\tikztostart)!#1!\pgfkeysvalueof{/tikz/ncbar angle}:(\tikztotarget)$)
        -- ($(\tikztotarget)!($(\tikztostart)!#1!\pgfkeysvalueof{/tikz/ncbar angle}:(\tikztotarget)$)!\pgfkeysvalueof{/tikz/ncbar angle}:(\tikztostart)$)
        -- (\tikztotarget)
    },
    ncbar/.default=0.5cm,
}
\tikzset{square left brace/.style={ncbar=-0.1cm}}
\tikzset{square right brace/.style={ncbar=0.1cm}}
\tikzset{round left paren/.style={ncbar=0.5cm,out=110,in=-110}}
\tikzset{round right paren/.style={ncbar=0.5cm,out=-70,in=70}}

\newtheoremstyle{slplain}% name
  {\topsep}%Space above
  {\topsep}%Space below
  {\slshape}%Body font
  {0pt}%Indent amount
  {\bfseries}% Theorem head font
  {.}%Punctuation after theorem head
  {0.5em}%Space after theorem head 2
  {}%Theorem head spec (can be left empty, meaning ‘normal’)

\theoremstyle{slplain}
  \newtheorem{THM}{Theorem}[section]
  \newtheorem{LEM}[THM]{Lemma}
  \newtheorem{PROP}[THM]{Proposition}
  \newtheorem{COR}[THM]{Corollary}

\theoremstyle{definition}

\newcommand{\End}{\operatorname{End}}

\newcommand{\Clo}{\operatorname{Clo}}
\newcommand{\Pol}{\operatorname{Pol}}

\newcommand{\ar} {\operatorname{ar}}

\newcommand{\ord}{\operatorname{ord}}

\newcommand{\type}{\operatorname{type}}
\let\le\leqslant
\let\leq\leqslant
\let\ge\geqslant
\let\geq\geqslant
\let\phi\varphi
\let\emptyset\varnothing
\newcommand{\pair}[2]{(#1,#2)}
\newcommand{\tuple}[1]{(#1)}

\title{On $k$-ary parts of maximal clones}

\author{
Dragan Ma\v{s}ulovi\'c, Maja Pech\\ Department of
Mathematics and Informatics\\ Faculty of Sciences, University of Novi Sad, Serbia\\
Trg Dositeja Obradovi\'ca 3, 21000 Novi Sad\\
e-mail: \{masul,maja\}@dmi.uns.ac.rs
}
\date{}
\begin{document}
\maketitle

\begin{abstract}
	The main problem of clone theory is to describe the clone lattice for a given basic set. For a two-element basic set this was resolved by E.L. Post, but for at least three-element basic set the full structure of the lattice is still unknown, and the complete description in general is considered to be hopeless. Therefore, it is studied by its substructures and its approximations. One of the possible directions is to examine $k$-ary parts of the clones and their mutual inclusions.\\ In this paper we study $k$-ary parts of maximal clones, for $k\geq2$, building on the already known results for their unary parts. It turns out that the poset of $k$-ary parts of maximal clones defined by central relations contains long chains.

  \medskip


  \medskip

\end{abstract}

\section{Preliminaries}

Throughout the paper we assume that $A$ is a finite set and
$|A|\geq 3$. Let $O_A^{(n)}$ denote the set of all $n$-ary
operations on $A$ (so that $O_A^{(1)}=A^A$) and let
$O_A:=\bigcup_{n\geq 1} O_A^{(n)}$ denote the set of all finitary
operations on $A$. For $F\subseteq O_A$ let $F^{(n)}:=F\cap
O_A^{(n)}$ be the set of all $n$-ary operations in $F$. A set
$C\subseteq O_A$ of finitary operations is a {\it clone of
operations on $A$} if it contains all projection maps ${\pi}^n_i:
A^n\rightarrow A: (x_1,\dots ,x_n)\mapsto x_i$ and is closed with
respect to composition of functions in the following sense:
  whenever $g\in C^{(n)}$ and $f_1,\dots ,f_n\in C^{(m)}$ for some
  positive integers $m$ and $n$ then $g(f_1,\dots ,f_n)\in C^{(m)}$,
  where the composition $h:=g(f_1,\dots ,f_n)$ is defined by
  $h(x_1,\dots ,x_m):= g(f_1(x_1,\dots ,x_m),\dots ,f_n(x_1,\dots
  ,x_m))$.

Clearly, for any family $(C_i)_{i\in I}$ of clones on $A$ we have that $\bigcap_{i\in I} C_i$ is a clone, too. Therefore, for any $F\subseteq O_A$ it makes sense to define $\Clo (F)$ to be the smallest clone that contains $F$.

We say that an $n$-ary operation $f$ preserves an $h$-ary relation
$\varrho$ if the following holds:
\[
\left [ \begin{array}{@{}c@{}} a_{11}\\ a_{21}\\ \vdots\\ a_{h1}
\end{array}\right],
\left [ \begin{array}{@{}c@{}} a_{12}\\ a_{22}\\ \vdots\\ a_{h2}
\end{array}\right],
\dots, \left [ \begin{array}{@{}c@{}} a_{1n}\\ a_{2n}\\ \vdots\\ a_{hn}
\end{array}\right]\in \varrho \mbox{ implies }
\left [ \begin{array}{@{}c@{}} f(a_{11},a_{12},\dots ,a_{1n})\\
f(a_{21},a_{22},\dots ,a_{2n}) \\ \vdots\\  f(a_{h1},a_{h2},\dots
,a_{hn})
\end{array}\right]\in \varrho.
\]
For a set $Q$ of relations let
\begin{eqnarray*}
\Pol Q &:=& \{ f\in O_A: f\mbox{ preserves every } \varrho\in
Q\}.
\end{eqnarray*}
Let $\Pol_n Q =(\Pol Q)\cap O_A^{(n)}$. For an $h$-ary relation
$\theta\subseteq A^h$ and a unary operation $f\in A^A$ it is
convenient to write
\[
  f(\theta ):= \{ \tuple{f(x_1),\dots , f(x_h)} : \tuple{x_1,\dots ,x_h}
  \in \theta \}.
\] Then clearly $f$ preserves $\theta$ if and only if $f(\theta
)\subseteq \theta$. It follows that $\Pol_1 Q$ is the
endomorphism monoid of the relational structure $\pair AQ$.
Therefore instead of $\Pol_1 Q$ we simply write $\End
Q$.  

If the underlying set is finite and has at least three elements,
then the lattice of clones has cardinality $2^{{\aleph}_0}$.
However, one can show that the lattice of clones on a finite set
has a finite number of coatoms, called {\it maximal clones}, and
that every clone distinct from $O_A$ is contained in one of the
maximal clones. One of the most influential results in clone
theory is the explicit characterization of the maximal clones,
obtained by I.\ G.\ Rosenberg as the culmination of the work of
many mathematicians. It is usually stated in terms of the
following six classes of finitary relations on $A$ (the so-called
{\it Rosenberg relations}).

\begin{enumerate}[label=(R\arabic*), ref=(R\arabic*)]
\item\label{r1} {\it Bounded partial orders.} These are partial orders
on $A$ with a least and a greatest element.

\item\label{r2} {\it Nontrivial equivalence relations.} These are
equivalence relations on $A$ distinct from ${\Delta}_A:=\{
\pair xx : x\in A\}$ and $A^2$.

\item\label{r3} {\it Permutational relations.} These are relations of
the form $\{ \pair{x}{\pi (x)} : x\in A\}$ where $\pi$ is a
fixpoint-free permutation of $A$ with all cycles of the same
length $p$, where $p$ is a prime.

\item\label{r4} {\it Affine relations.} For a binary operation
$\oplus$ on $A$ let
\[{\lambda}_{\oplus}:=\{ \tuple{x,y,u,v} \in A^4 :
x\oplus y=u\oplus v\} .\]

A relation $\varrho$ is called {\it affine} if there is an
elementary abelian $p$-group $\tuple{A,\oplus ,\ominus ,0}$
on $A$ such that $\varrho={\lambda}_{\oplus}$.

Suppose now that $A$ is an elementary abelian $p$-group. Then it
is well-known that $f\in \Pol \{ {\lambda}_{\oplus}\}$ if and only
if \[f(x_1\oplus y_1,\dots ,x_n\oplus y_n)=f(x_1,\dots ,x_n)\oplus
f(y_1,\dots , y_n)\ominus f(0,\dots ,0)\]

for all $x_i,y_i\in A$. In case $f$ is unary, this condition
becomes \[f(x\oplus y)=f(x)\oplus f(y)\ominus f(0).\]

\item\label{r5} {\it Central relations.} All unary relations are
central relations. For central relations $\varrho$ of arity $h\geq
2$ the definition is as follows: $\varrho$ is said to be {\it
totally symmetric} if $\tuple{x_1,\dots ,x_h}\in \varrho$ implies
$\tuple{x_{\pi (1)},\dots ,x_{\pi (h)}}\in \varrho$ for all permutations
$\pi$, and it is said to be {\it totally reflexive} if
$\tuple{x_1,\dots,x_h}\in \varrho$ whenever there are $i \ne j$ such that
$x_i=x_j$. An element $c\in A$ is {\it central} if  $\tuple{c,x_2,\dots
,x_h} \in \varrho$ for all $x_2,\dots ,x_h\in A$. Finally,
$\varrho\not= A^h$ is called {\it central} if it is totally
reflexive, totally symmetric and has a central element. According
to this, every central relation $\varrho$ can be written as $C_{\varrho}\cup R_{\varrho}\cup T_{\varrho}$,
where $C_{\varrho}$ consists of all the tuples of distinct elements
containing at least one central element (the central part), $R_{\varrho}$
consists of all the tuples $\tuple{x_1,\dots ,x_h}$ such that there are
$i\not= j$ with $x_i=x_j$ (the reflexive part) and $T_{\varrho}$ consists of
all the tuples $\tuple{x_1,\dots ,x_h}$ such that $x_1,\dots ,x_h$ are
distinct non-central elements. Let $Z_\varrho$ denote the set
of all central elements of $\varrho$.

\item\label{r6} {\it $h$-regular relations.} Let $\Theta =\{
{\theta}_1,\dots ,{\theta}_m\}$ be a family of equivalence
relations on the same set $A$. We say that $\Theta$ is an {\it $h$-regular family} if
every ${\theta}_i$ has precisely $h$ blocks, and additionally, if
$B_i$ is an arbitrary block of ${\theta}_i$ for $i\in \{ 1,\dots
,m\}$, then $\bigcap_{i=1}^m B_i\not=\emptyset .$

An $h$-ary relation $\varrho\not= A^h$ is $h$-regular if $h\geq 3$
and there is an $h$-regular family $\Theta$ such that $\tuple{x_1,\dots
,x_h}\in \varrho$ if and only if for all $\theta\in \Theta$ there
are distinct $i$, $j$ with $x_i \theta x_j$. Clearly, $\varrho$ is completely determined by its $h$-regular family $\Theta$. Therefore, we will also denote it by $R_{\Theta}$.

Note that regular relations are totally reflexive and totally
symmetric.
\end{enumerate}

\begin{THM}[Rosenberg \cite{Rosen}]
  A clone $M$ of operations on a finite set is maximal if and only
  if there is a relation $\varrho$ from one of the classes
  \ref{r1}--\ref{r6} such that $M=\Pol\{\varrho\}$.
\end{THM}

\begin{table}
\centering
\let\gc\relax
\newcommand{\fd}[2]{{#1}\vfill{\centerline{#2}}\vfill}
{\scriptsize
\begin{tabular}{|p{12mm}|p{12mm}|p{12mm}|p{12mm}|p{12mm}|p{12mm}|p{12mm}|p{12mm}|}\hline
 \begin{center} {\Large $ \varrho \backslash \sigma$ }\end{center}
 &\vspace{1mm} Bounded partial order  &\vspace{1mm} Equiva\-lence relation &\vspace{1mm}
  Permuta\-ti\-o\-nal relation &\vspace{1mm} Affine relation &\vspace{1mm} Unary central relation &
 \vspace{1mm} $k$--ary central relation,
  $k\geq 2$ &\vspace{1mm} $h$--regular relation\\ \hline
\vspace{1mm}  Bounded partial order & \fd{\large \[-\]}{\cite{MPDM3}}
& \fd{\large
\[-\]}{\cite{MPDM3}} & \fd{\large \[-\]}{\cite{MPDM3}} & \fd{\large \[-\]}{\cite{MPDM3}} & \fd{\large \[-\]}{\cite{MPDM3}} &
\gc{\fd{\large \[+ ?\]}{\cite{MP10}}} & \fd{\large \[+ ?\]}{\cite{MP10}}\\
\hline \vspace{1mm} Equiva\-lence relation & \fd{\large
\[-\]}{\cite{MPDM3}} & \fd{\large \[-\]}{\cite{MPDM3}} & \fd{\large \[-\]}{\cite{MPDM3}} & \fd{\large \[-\]}{\cite{MPDM3}} &
 \fd{\large \[-\]}{\cite{MPDM3}} & \gc{\fd{\large \[-\]}{\cite{MPDM1}}}& \gc{\fd{\large
\[+\]}{\cite{MPDM2}}}\\ \hline \vspace{1mm}  Permuta\-ti\-o\-nal relation &
\fd{\large \[-\]}{\cite{MPDM3}} & \fd{\large \[+\]}{\cite{MPDM3}} &
\fd{\large \[-\]}{\cite{MPDM3}} & \fd{\large \[+\]}{\cite{MPDM3}} & \fd{\large \[-\]}{\cite{MPDM3}} &
\gc{\fd{\large \[-\]}{\cite{MPDM1}}} & \fd{\large \[+\]}{\cite{MPDM3}}\\
\hline \vspace{1mm} Affine relation & \fd{\large \[-\]}{\cite{MPDM3}}
& \fd{\large \[-\]}{\cite{MPDM3}} & \fd{\large \[-\]}{\cite{MPDM3}} &
\fd{\large \[-\]}{\cite{MPDM3}} & \fd{\large
\[-\]}{\cite{MPDM3}} & \gc{\fd{\large \[-\]}{\cite{MPDM1}}} & \fd{\large \[+\]}{\cite{MPDM3}}\\ \hline
\vspace{1mm} Unary central relation & \gc{\fd{\large
\[-\]}{\cite{MPDM1}}} & \gc{\fd{\large \[+\]}{\cite{MPDM1}}} & \gc{\fd{\large
\[-\]}{\cite{MPDM1}}} & \gc{\fd{\large \[-\]}{\cite{MPDM1}}} & \gc{\fd{\large \[-\]}{\cite{MPDM1}}} &
\gc{\fd{\large \[+\]}{\cite{MPDM1}}} & \gc{\fd{\large \[+\]}{\cite{MPDM1}}}\\
\hline \vspace{1mm}   $k$--ary central relation, $k\geq 2$ &
\gc{\fd{\large
\[-\]}{\cite{MPDM1}}} & \gc{\fd{\large \[+\]}{\cite{MPDM1}}} & \gc{\fd{\large
\[-\]}{\cite{MPDM1}}} & \gc{\fd{\large \[-\]}{\cite{MPDM1}}} & \gc{\fd{\large \[-\]}{\cite{MPDM1}}} &
 \gc{\fd{\large \[+\]}{\cite{MP2}}} &
\gc{\fd{\large \[+\]}{\cite{MP10}}}\\ \hline
\vspace{1mm} $h$--regular relation & \fd{\large \[-\]}{\cite{MPDM3}}
& \fd{\large \[-\]}{\cite{MPDM3}} & \fd{\large \[-\]}{\cite{MPDM3}} & \fd{\large \[-\]}{\cite{MPDM3}}
& \fd{\large \[-\]}{\cite{MPDM3}} & {\fd{\large \[+\]}{\cite{MP10}}} & \gc{\fd{\large \[+\]}{\cite{MP10}}} \\
\hline
\end{tabular}
}
\caption{A summary of the results}
\label{table.overview}
\end{table}

Table~\ref{table.overview} summarizes all known results about the mutual containment of unary parts of maximal clones over a finite set $A$ with $|A|\geq 3$.
The entries in this table are to be interpreted in the following way:
\begin{itemize}
\item
    we write $-$ if $\End\varrho \not\subseteq
    \End\sigma$ for every pair $\pair\varrho\sigma$ of distinct relations
    of the indicated type;

\item
    we write $+$ if there is a complete characterization of the
    situation $\End\varrho \subseteq \End\sigma$;

\item
    we write $+ ?$  if there is a partial characterization of the
    situation $\End\varrho \subseteq \End\sigma$.
\end{itemize}

\section{Binary operations in maximal clones}

The partially ordered set of unary parts of maximal clones ordered by inclusion has a very rich
structure \cite{MPDM1,MPDM2,MPDM3}. Moreover, the main result of \cite{MP2} shows that
every finite Boolean algebra is order-embeddable into the partially ordered set of unary
parts of maximal clones on a sufficiently large finite set. In this section we would like
to consider a similar problem and embark on the investigation of the partially ordered set
$\{\Pol_2 \varrho : \varrho$ is a Rosenberg relation$\}$.

This problem is closely related to the notion of the order of a clone.
For a finitely generated clone $C$, let $\ord(C)$ denote the \emph{order of $C$},
that is, the least positive integer $k$ such that $\Clo(C^{(k)}) = C$.
If $C$ is not finitely generated we set $\ord(C) = \infty$.
It is easy to see that if $C$ is a maximal clone with $\ord(C) = 2$ and $D$ is another
maximal clone then $C^{(2)} \not\subseteq D^{(2)}$, for otherwise we would have
$C = \Clo(C^{(2)}) \subseteq \Clo(D^{(2)}) \subseteq D$, which contradicts the maximality
of~$C$.

\begin{PROP}\label{Pol2.prop.main}
  If $\varrho$ and $\sigma$ are distinct Rosenberg relations such that
  $\Pol_2 \varrho \subseteq \Pol_2 \sigma$
  then both $\varrho$ and $\sigma$ have to be central relations of arity at least~2.
\end{PROP}
\begin{proof}
  It is a well-known fact (see \cite{Poe-Kal}) that if $|A| \ge 3$ then $\ord(C) = 2$ for
  all maximal clones $C = \Pol \varrho$ where $\varrho$ belongs to one of the classes
  \ref{r2}, \ref{r3}, \ref{r4} and \ref{r6}. Therefore, if $\Pol_2 \varrho \subseteq \Pol_2 \sigma$
  for some Rosenberg relations $\varrho$ and $\sigma$, then $\varrho$ is a
  bounded partial order or a central relation.

  \bigskip

  \textsc{Step 1.}
  Let $\varrho$ be a bounded partial order with the least element $0$ and the greatest element $1$.
  If $\sigma$ belongs to one of the classes \ref{r1}, \ref{r2}, \ref{r3}, \ref{r4} or if $\sigma$ is
  a unary central relation, then $\End \varrho \not\subseteq \End\sigma$
  (see Table~\ref{table.overview}), and hence $\Pol_2 \varrho \not\subseteq \Pol_2 \sigma$.

  Let $\sigma$ be a central relation of arity $k \ge 2$ and consider the following three
  binary operations on $A$:
  \[
    f(x, y) = \begin{cases}
      x, & \text{if }y = 1,\\
      y, & \text{if }x = 1,\\
      0, & \text{otherwise},
    \end{cases}
    \mbox{\qquad}
    g(x, y) = \begin{cases}
      x, & \text{if }y = 0,\\
      y, & \text{if }x = 0,\\
      1, & \text{otherwise},
    \end{cases}
  \]
  \[
    \text{and\quad}
    t_{a,b}(x, y) = \begin{cases}
      0, & \text{if }(x, y) \le (a, b),\\
      x, & \text{otherwise}.
    \end{cases}
  \]
  All three operations are monotonous with respect to $\varrho$ and $f(1, x) = f(x, 1) = g(0, x) =
  g(x, 0) = x$. If $1 \in Z_\sigma$, take any $\tuple{x_1, x_2, \ldots, x_k} \notin \sigma$
  and note that
  \[
    \begin{array}{clll@{\,}l}
      (\; 1, & x_2, & \ldots, & x_k)   &\in \sigma\\
      ( x_1, & 1, & \ldots, & 1\;)     &\in \sigma\\
      \llap{$f:\,$}\downarrow & \downarrow & \ldots & \downarrow\\
      ( x_1, & x_2, & \ldots, & x_k) &\notin \sigma.
    \end{array}
  \]
  Thus, $f$ does not preserve $\sigma$, and $\Pol_2 \varrho \not\subseteq \Pol_2 \sigma$.
  
  If $0 \in Z_\sigma$, take any $\tuple{x_1, x_2, \ldots, x_k} \notin \sigma$
  and note that
 \[
    \begin{array}{clll@{\,}l}
      (\; 0, & x_2, & \ldots, & x_k)   &\in \sigma\\
      ( x_1, & 0, & \ldots, & 0\;)     &\in \sigma\\
      \llap{$g:\,$}\downarrow & \downarrow & \ldots & \downarrow\\
      ( x_1, & x_2, & \ldots, & x_k) &\notin \sigma.
    \end{array}
  \]
   Thus, $g$ does not preserve $\sigma$, and $\Pol_2 \varrho \not\subseteq \Pol_2 \sigma$.
   
  Finally, assume that $0 \notin Z_\sigma$ and $1 \notin Z_\sigma$. Since $0 \notin Z_\sigma$
  there exist $x_2$, \ldots, $x_k \in A$ such that $\tuple{0, x_2, \ldots, x_k} \notin \sigma$.
  Take any $c \in Z_\sigma$ and note that $t_{c,c}(c, c) = 0$ and $t_{c,c}(x_i, 1) = x_i$
  since $c < 1$. Therefore,
  \[
    \begin{array}{clll@{\,}l}
      ( c, & x_2, & \ldots, & x_k)   &\in \sigma\\
      ( c, & 1, & \ldots, & 1\;)       &\in \sigma\\
      \llap{$t_{c,c}:\,$}\downarrow & \downarrow & \ldots & \downarrow\\
      ( 0, & x_2, & \ldots, & x_k) &\notin \sigma.
    \end{array}
  \]
   Thus, $t_{c,c}$ does not preserve $\sigma$, and $\Pol_2 \varrho \not\subseteq \Pol_2 \sigma$.
   
  This completes the proof that if $\varrho$ is a bounded partial order and $\sigma$ is a
  central relation then $\Pol_2 \varrho \not\subseteq \Pol_2 \sigma$.

  Now, let $\sigma=R_{\Theta}$ be a regular relation. From \cite[Proposition 4.25]{MPDM3}
  we know that if $\End\varrho \subseteq \End R_\Theta$ where $\varrho$ is a bounded partial
  order, then $\Theta$ has to be a singleton $\Theta = \{\theta\}$. Let $B_1$, \ldots, $B_h$
  be the blocks of $\theta$. One of the $B_i$'s contains $0$, so without loss of generality
  we can assume that $0 \in B_1$. If $1$ is not the only element in its block, we can
  choose $x_2 \in B_2$, \ldots, $x_h \in B_h$ such that $1 \notin \{x_2, \ldots, x_h\}$.
  But then
  \[
    \begin{array}{cllll@{\,}l}
      ( x_2, & x_2, &  x_3, & \ldots, & x_h)   &\in R_\Theta\\
      ( x_2, & 1,   & 1,    & \ldots, & 1\;)     &\in R_\Theta\\
      \llap{$t_{x_2,x_2}:\,$}\downarrow & \downarrow & \downarrow & \ldots & \downarrow\\
      ( 0,   & x_2, & x_3,  & \ldots, & x_h)   &\notin R_\Theta,
    \end{array}
 \]
  so, $\Pol_2 \varrho \not\subseteq \Pol_2 R_\Theta$. If $1$ is the only element in its block,
  without loss of generality we can assume $B_2 = \{1\}$. Take arbitrary
  $x_3 \in B_3$, \ldots, $x_h \in B_h$ and note that
  \[
    \begin{array}{cllll@{\,}l}
      ( x_3, & 1, &  x_3, & \ldots, & x_h)   &\in R_\Theta\\
      ( x_3, & 1,   & 1,    & \ldots, & 1\;)     &\in R_\Theta\\
      \llap{$t_{x_3,x_3}:\,$}\downarrow & \downarrow & \downarrow & \ldots & \downarrow\\
      ( 0,   & 1, & x_3,  & \ldots, & x_h)   &\notin R_\Theta.
    \end{array}
  \]
  Therefore, $\Pol_2 \varrho \not\subseteq \Pol_2 R_\Theta$. This completes the proof that
  $\varrho$ cannot be a bounded partial order if $\Pol_2 \varrho \subseteq \Pol_2 \sigma$.

  \bigskip

  \textsc{Step 2.}
  Let $\varrho$ be a central relation.
  If $\sigma$ belongs to one of the classes \ref{r1}, \ref{r3}, \ref{r4} or if $\sigma$ is
  a unary central relation, then $\End \varrho \not\subseteq \End\sigma$
  (see Table~\ref{table.overview}), and hence $\Pol_2 \varrho \not\subseteq \Pol_2 \sigma$.

  Suppose $\sigma$ is an equivalence relation. According to \cite[Proposition 4.3]{MPDM1},
  from $\End \varrho \subseteq \End \sigma$ it follows that $\ar(\varrho) \in \{1, 2\}$,
  $T_{\varrho}=\emptyset$ and $A / \sigma = \{Z_\varrho, \{a_2\}, \ldots, \{a_t\}\}$,
  i.e.~$Z_\varrho$ is the only nontrivial block of $\sigma$. Since $\sigma$ is a nontrivial
  equivalence relation we have that
  $|Z_\varrho| \ge 2$ and $t \ge 2$. Take $c_1, c_2 \in Z_\varrho$ so that $c_1 \ne c_2$ and
  define $* : A^2 \to A$ by $c_1 * y = c_1$ and $x * y = y$ for $x \ne c_1$.
  Clearly, $* \in \Pol_2 \varrho$. To see that $* \notin \Pol_2 \sigma$, note that
  $\pair{c_1}{c_2} \in \sigma$ and $\pair{a_2}{a_2} \in \sigma$ but
  $\pair{c_1 * a_2}{c_2 * a_2} = \pair{c_1}{a_2} \notin \sigma$. Therefore,
  $\Pol_2 \varrho \not\subseteq \Pol_2 \sigma$.

  Suppose $\sigma$ is a regular relation defined by an $h$-regular family $\Theta$.
  According to \cite[Propositions 4.6 and 4.7]{MPDM1} from $\End \varrho \subseteq \End \sigma$
  it follows that $\Theta = \{\theta\}$, $A / \theta = \{B, \{b_2\},
  \ldots, \{b_h\}\}$, $|B| \ge 2$ and $Z_\varrho \subseteq B$.
  Define $* : A^2 \to A$ by $x * y = y$ if $y \in Z_\varrho$
  and $x * y = x$ otherwise. Then clearly $* \in \Pol_2 \varrho$.
  To see that $* \notin \Pol_2 \sigma$ take arbitrary $c \in Z_\varrho$ and note that
  \[
    \begin{array}{cllll@{\,}l}
      ( b_2,   & b_2, & b_3, & \ldots, & b_h )   &\in \sigma\\
      ( \;c, & b_2, & b_2, & \ldots, & b_2 )   &\in \sigma\\
      \llap{$*:\,$}\downarrow & \downarrow & \downarrow & \ldots & \downarrow\\
      ( \;c, & b_2, & b_3, & \ldots, & b_h )   &\notin \sigma.
    \end{array}
\]

  If $\varrho$ is a unary central relation and $\sigma$ is an at least binary central
  relation, say of arity $k$, then according to
  \cite[Proposition 4.1]{MPDM1}, $Z_\sigma = \varrho$ and $T_{\sigma}=\emptyset$.
  Define $* : A^2 \to A$ by $x * y = y$ if $x \in \varrho$ and $x * y = x$ otherwise.
  Clearly, $* \in \Pol_2 \varrho$. To see that $* \notin \Pol_2 \sigma$ take any
  $c \in \varrho = Z_\sigma$ and any $\tuple{x_1, x_2, \ldots, x_k} \notin \sigma$.
  Then
  \[
    \begin{array}{clll@{\,}l}
      (\;\; c,   & x_2, & \ldots, & x_k)     &\in \sigma\\
      ( x_1, & c,   & \ldots, & c\;\;)     &\in \sigma\\
      \llap{$*:\,$}\downarrow & \downarrow & \ldots & \downarrow\\
      ( x_1, & x_2, & \ldots, & x_k)   &\notin \sigma.
    \end{array}
 \]

  Therefore, if $\Pol_2 \varrho \subseteq \Pol_2 \sigma$ then both $\varrho$ and $\sigma$ have to
  be at least binary central relations.
\end{proof}

At this point it is clear that the only nontrivial containments among $k$-ary parts of maximal clones with $k\geq 2$ can occur for central relations of arity at least 2. This case is studied in detail in the next section.

\section{Rosenberg clones defined by central relations}

To untangle the situation concerning the $k$-ary parts of Rosenberg clones of central relations
we introduce another set of strategies.
Let $\varrho$ be an $n$-ary relation on $A$ and let $\bar a = (a_1,\dots,a_m) \in A^m$. Let us define
the \emph{type of $\bar a$ with respect to $\varrho$} as follows:
\[
  {\type}_{\varrho} (\bar{a}) = {\type}_{\varrho}(a_1,\dots,a_m)=({\tau}_1(\bar{a}),{\tau}_2(\bar{a}))
\]
where
\begin{align*}
{\tau}_1(\bar{a}) =& \{ (i_1,\dots,i_n)\mid i_1<\dots <i_n\in
\{1,\dots ,m\}\text{ and } (a_{i_1},\dots,a_{i_n})\in\varrho\}, \text{ and}\\
{\tau}_2(\bar{a}) =& \{ (i,j)\mid i<j\in \{1,\dots ,m\}, \;
 \text{ and } a_i=a_j\}.
\end{align*}
For $\bar{a}_1,\bar{a}_2\in A^m$ define
\[{\type}_{\varrho} (\bar{a}_1)\cap {\type}_{\varrho} (\bar{a}_2):= ({\tau}_1(\bar{a}_1)\cap {\tau}_1(\bar{a}_2) ,{\tau}_2(\bar{a}_1)\cap {\tau}_2(\bar{a}_2)),\]
and ${\type}_{\varrho} (\bar{a})\subseteq {\type}_{\varrho} (\bar{b})$ if ${\tau}_i (\bar{a})\subseteq \tau_i(\bar{b})$ for $i=1,2$.

\begin{PROP}\label{Polcentralnih}
  Let $\varrho$ be an $n$-ary central relation on $A$ and let $\sigma$
  be any $m$-ary relation on the same set $A$. Then $\Pol_k \varrho \subseteq \Pol_k \sigma$ if and only if the following holds for every
  $\bar{a}_1,\dots ,\bar{a}_k\in \sigma$ and every $\bar{b}\in A^m$:
  \[
  {\type}_{\varrho} (\bar{a}_1)\cap {\type}_{\varrho}
  (\bar{a}_2)\cap\dots \cap {\type}_{\varrho} (\bar{a}_k)\subseteq
  {\type}_{\varrho} (\bar{b})\Rightarrow \bar{b}\in\sigma .
  \]
\end{PROP}
\begin{proof}
  ($\Leftarrow$)  Assume that for every ${\bar a}_1,\dots ,{\bar a}_k\in
  \sigma$ and every $\bar{b}\in A^m$ we have that
  $
  {\type}_{\varrho} ({\bar a}_1)\cap {\type}_{\varrho}
  ({\bar a}_2)\cap\dots \cap {\type}_{\varrho} ({\bar a}_k)\subseteq
  {\type}_{\varrho} (\bar{b})\Rightarrow \bar{b}\in\sigma
  $.

  Take $f\in \Pol_k \varrho$ and ${\bar a}_1,\dots ,{\bar a}_k\in
  \sigma$, say, \begin{center}  \begin{tikzpicture}[
anno/.style={column #1/.style={nodes={text=red, font=\tiny\ttfamily}}},
anno/.list={5}
]
\matrix (m) [matrix of math nodes,  
inner sep=0pt, column sep=0.25em, 
nodes={inner sep=0.25em,text width=2em,align=center}
]
{
\bar{a}_1 & \bar{a}_2 & \cdots & \bar{a}_k\\
\rotatebox[origin=c]{90}{=} & \rotatebox[origin=c]{90}{=} & &\rotatebox[origin=c]{90}{=} \\
a_1^1  &  a_1^2 &  \cdots & a_1^k \\
a_2^1  &  a_2^2 &  \cdots & a_2^k \\
\vdots &  \vdots  &  \ddots & \vdots \\
a_m^1  &  a_m^2 &  \cdots & a_m^k \\
};
\draw (m-3-1.north west) to [square left brace] (m-6-1.south west) ;
\draw(m-3-4.north east) to [square right brace](m-6-4.south east);
\end{tikzpicture} .\end{center} We define $\bar{b}$ in the following way:
  $\bar{b}=f({\bar a}_1,\dots ,{\bar a}_k)$, i.e.\ $b_i=f(a_i^1,\dots,
  a_i^k)$, $i\in\{ 1,\dots ,m\}$. We will show that $\bar{b}\in
  \sigma$.
  According to the assumption, it suffices to show that
  \[
    {\type}_{\varrho} ({\bar a}_1)\cap {\type}_{\varrho}
    ({\bar a}_2)\cap\dots \cap {\type}_{\varrho} ({\bar a}_k)\subseteq
    {\type}_{\varrho} (\bar{b})
  \]
  or, equivalently,
  \begin{gather*}
    {\tau}_{1} ({\bar a}_1)\cap {\tau}_{1} ({\bar a}_2)\cap\dots \cap {\tau}_{1}
    ({\bar a}_k)\subseteq {\tau}_{1} (\bar{b})\\
    \text{ and } {\tau}_{2}
    ({\bar a}_1)\cap {\tau}_{2} ({\bar a}_2)\cap\dots \cap {\tau}_{ 2}
    ({\bar a}_k)\subseteq {\tau}_{2} (\bar{b}).
  \end{gather*}

  For the first inclusion take any $(i_1,\dots ,i_n)\in {\tau}_{1}
  ({\bar a}_1)\cap {\tau}_{1} ({\bar a}_2)\cap\dots \cap {\tau}_{1}
  ({\bar a}_k)$. Then
  $
    (a_{i_1}^1,\dots ,a_{i_n}^1),\dots
    ,(a_{i_1}^k,\dots ,a_{i_n}^k)\in\varrho.
  $
  Since $f\in \Pol_k \varrho$ it follows that
  \[
  \left [ \begin{array}{c} f(a_{i_1}^1,a_{i_1}^2,\dots ,a_{i_1}^k)\\
  f(a_{i_2}^1,a_{i_2}^2,\dots ,a_{i_2}^k) \\ \vdots\\
  f(a_{i_n}^1,a_{i_n}^2,\dots ,a_{i_n}^k)
  \end{array}\right]\in \varrho ,
  \]
  i.e.,  $(b_{i_1},b_{i_2},\dots ,b_{i_n})\in \varrho$, so
  $(i_1,i_2,\dots ,i_n)\in {\tau}_{1} (\bar{b})$.
  
  For the second inclusion let $(i,j)\in {\tau}_{2}
  ({\bar a}_1)\cap {\tau}_{2} ({\bar a}_2)\cap\dots \cap {\tau}_{ 2}
  ({\bar a}_k)$. Then $a_i^l=a_j^l$, $l=1,\dots,k$. It follows that $(i,j)\in {\tau}_{2} (\bar{b})$
  since
  \[
  b_i=f(a_i^1,\dots,a_i^k)=f(a_j^1,\dots,a_j^k)=b_j.
  \]
  Putting it all together, $\bar{b}\in \sigma $ and, therefore, $f\in \Pol_k \sigma$.
  
  ($\Rightarrow $) Assume $\Pol_k \varrho \subseteq \Pol_k \sigma$. Take ${\bar a}_1,\dots ,{\bar a}_k\in \sigma$ and $\bar{b}\in A^m$, say,
  \begin{center}  \begin{tikzpicture}[
anno/.style={column #1/.style={nodes={text=red, font=\tiny\ttfamily}}},
anno/.list={5}
]
\matrix (m) [matrix of math nodes,  
inner sep=0pt, column sep=0.25em, 
nodes={inner sep=0.25em,text width=2em,align=center}
]
{
\bar{a}_1 & \bar{a}_2 & \cdots & \bar{a}_k & & \bar{b}\\
\rotatebox[origin=c]{90}{=} & \rotatebox[origin=c]{90}{=} &  &\rotatebox[origin=c]{90}{=} & & \rotatebox[origin=c]{90}{=}\\
a_1^1  &  a_1^2 &  \cdots & a_1^k & & b_1 \\
a_2^1  &  a_2^2 &  \cdots & a_2^k & & b_2\\
\vdots &  \vdots  &  \ddots & \vdots & \textrm{\normalsize and} & \vdots\\
a_m^1  &  a_m^2 &  \cdots & a_m^k & & b_m \\
};
\draw (m-3-1.north west) to [square left brace] (m-6-1.south west) ;
\draw(m-3-4.north east) to [square right brace](m-6-4.south east);
\draw (m-3-6.north west) to [square left brace] (m-6-6.south west) ;
\draw(m-3-6.north east) to [square right brace](m-6-6.south east);
\end{tikzpicture}
   \end{center}
  such that
  \[
  {\type}_{\varrho} ({\bar a}_1)\cap {\type}_{\varrho}
  ({\bar a}_2)\cap\dots \cap {\type}_{\varrho} ({\bar a}_k)\subseteq
  {\type}_{\varrho} (\bar{b}).
  \]
  We shall now construct an $f\in \Pol_k \sigma$ such that $f({\bar a}_1,\dots
  ,{\bar a}_k)=\bar{b}$, in the following way:
  
  \[
  f(x_1,\dots,x_k )=\left\{
  \begin{array}{ll}
   b_i, &\mbox{ if } (x_1,\dots ,x_k)=(a_i^1,\dots ,a_i^k), i=1,\dots ,m\\
   c, & \mbox{ otherwise,}
  \end{array}\right.
  \]
  where $c\in Z_{\varrho}$.
  Clearly, $f({\bar a}_1,\dots ,{\bar a}_k)=\bar{b}$, so it is left to
  show that $f$ is well defined and that $f\in \Pol_k \varrho$
  (and, therefore, $f\in \Pol_k \sigma$).
  
  To see that $f$ is well defined, suppose that $(a_i^1,\dots,a_i^k)=(a_j^1,\dots,a_j^k)$
  for some $i\not=j$, then $(i,j)\in
  {\tau}_{2} ({\bar a}_1)\cap {\tau}_{2} ({\bar a}_2)\cap\dots \cap
  {\tau}_{ 2} ({\bar a}_k)\subseteq {\tau}_{2} (\bar{b})$, so
  $b_i=b_j$, and $f$ is indeed well defined.
 
  To see that $f\in \Pol_k \varrho$ let ${\bar x}_1,\dots,{\bar x}_k\in \varrho$, where
  $\bar{x}_i=(x_1^i,\dots ,x_n^i)$, $i=1,\dots ,k$. Then
  \[
  f({\bar x}_1,\dots,{\bar x}_k)=\left [ \begin{array}{c} f(x_{1}^1,x_{1}^2,\dots ,x_{1}^k)\\
  f(x_{2}^1,x_{2}^2,\dots ,x_{2}^k) \\ \vdots\\
  f(x_{n}^1,x_{n}^2,\dots ,x_{n}^k)
  \end{array}\right] .
  \]
  If there is $(x_i^1,\dots ,x_i^k)\not=(a_j^1,\dots ,a_j^k)$, for
  some $j\in \{1,\dots ,m\}$, then $f(x_i^1,\dots ,x_i^k)=c$, so
  $f({\bar x}_1,\dots,{\bar x}_k)$ is a tuple that contains a central
  element, and, therefore, it is in $\varrho$.
  
  Otherwise, for each $i\in \{1,\dots,n\}$ $(x_i^1,\dots ,x_i^k)=(a_{j_i}^1,\dots ,a_{j_i}^k)$, for some
  $j_i\in \{1,\dots ,m\}$.
  
  If $(x_{i_1}^1,\dots ,x_{i_1}^k)=(x_{i_2}^1,\dots
  ,x_{i_2}^k)=(a_j^1,\dots ,a_j^n)$, where $i_1,i_2\in \{1,\dots ,n\}$
  and $i_1\not=i_2$, then $f({\bar x}_1,\dots,{\bar x}_k)$ is a reflexive tuple,
  so it belongs to $\varrho$.
  
  If that fails to be true then
  \[
  f({\bar x}_1,\dots,{\bar x}_k)=\left [ \begin{array}{c} f(a_{j_1}^1,a_{j_1}^2,\dots ,a_{j_1}^k)\\
  f(a_{j_2}^1,a_{j_2}^2,\dots ,a_{j_2}^k) \\ \vdots\\
  f(a_{j_n}^1,a_{j_n}^2,\dots ,a_{j_n}^k)
  \end{array}\right] =\left [ \begin{array}{c} b_{j_1}\\b_{j_2}\\ \vdots \\ b_{j_n}\end{array}\right].
  \]
  Since $\bar{x}_i=(a_{j_1}^i,\dots ,a_{j_n}^i)$ and $\bar{x}_i\in
  \varrho$, for $1\leq i \leq n$,  it follows that $(j_1,\dots
  ,j_n)\in {\tau}_{1} ({\bar a}_1)\cap {\tau}_{1} ({\bar a}_2)\cap\dots
  \cap {\tau}_{1} ({\bar a}_k)$, so $(j_1,\dots ,j_n)\in {\tau}_{1}
  (\bar{b})$, so $(b_{j_1},\dots ,b_{j_n})\in \varrho$.
   
  Therefore, $f\in \Pol_k \varrho \subseteq \Pol_k \sigma$, so $\bar{b}\in \sigma$.
\end{proof}

  \begin{LEM}\label{notail}
    Let $\varrho$ and $\sigma$ be two distinct central relations. If $\Pol_k \varrho \subseteq \Pol_k \sigma$ and $T_\varrho=\emptyset$, then $\ar (\varrho ) <\ar (\sigma )$, $Z_{\varrho}=Z_{\sigma}$ and $T_{\sigma}=\emptyset$.
  \end{LEM}
  
  \begin{proof}
    If $\Pol_k \varrho \subseteq \Pol_k \sigma$, then, clearly, ${\End} \varrho \subseteq {\End} \sigma$ and the claim follows from the corresponding lemma for endomorphisms, see \cite[Proposition 4.1]{MPDM1}.
  \end{proof}

  \begin{THM}\label{main-thm}
    Let $\varrho$ and $\sigma$ be two distinct central relations on $A$ such that $T_\varrho=\emptyset$. Then $\Pol_k \varrho \subseteq \Pol_k \sigma$ for $k\geq 2$ if and only if $2k\leq \ar(\varrho)<\ar(\sigma)\leq |A|-1$, $Z_\varrho=Z_\sigma$, and $T_\sigma=\emptyset$.
  \end{THM}
  \begin{proof}
 
  Let $\varrho\subseteq A^n$ and $\sigma\subseteq A^m$ be distinct central relations.
  
   $(\Leftarrow)$ Assume that $Z_\varrho = Z_\sigma$, and $T_\varrho=T_\sigma = \emptyset$, $2k\leq n<m$. We will show that $\Pol_k \varrho \subseteq \Pol_k \sigma$ using the criterion from Proposition~\ref{Polcentralnih}.
  
  Let $\bar{a}_1,\bar{a}_2,\dots, \bar{a}_k\in \sigma$, say,
\begin{center}  \begin{tikzpicture}[
anno/.style={column #1/.style={nodes={text=red, font=\tiny\ttfamily}}},
anno/.list={5}
]
\matrix (m) [matrix of math nodes,  
inner sep=0pt, column sep=0.25em, 
nodes={inner sep=0.25em,text width=2em,align=center}
]
{
\bar{a}_1 & \bar{a}_2 & \cdots & \bar{a}_k\\
\rotatebox[origin=c]{90}{=} & \rotatebox[origin=c]{90}{=} &  &\rotatebox[origin=c]{90}{=} \\
a_1^1  &  a_1^2 &  \cdots & a_1^k \\
a_2^1  &  a_2^2 &  \cdots & a_2^k \\
\vdots &  \vdots  &  \ddots & \vdots \\
a_m^1  &  a_m^2 &  \cdots & a_m^k \\
};
\draw (m-3-1.north west) to [square left brace] (m-6-1.south west) ;
\draw(m-3-4.north east) to [square right brace](m-6-4.south east);
\end{tikzpicture} .
\end{center}  

Let $\bar{b}\in A^m$ such that ${\type}_{\varrho} (\bar{a}_1)\cap {\type}_{\varrho}
  (\bar{a}_2)\cap\dots \cap {\type}_{\varrho} (\bar{a}_k)\subseteq
  {\type}_{\varrho} (\bar{b})$. According to Proposition~\ref{Polcentralnih} we have to show that $\bar{b}\in \sigma$. Note that if $\tau_2(\bar{b})\neq \emptyset$, then $\bar{b}\in \sigma$, and we are done. So suppose that $\tau_2(\bar{b})=\emptyset$. Let $J:=\{j\in \{1,\dots,k\}\mid \bar{a}_j\in C_\sigma\}$, $L:=\{l\in \{1,\dots,k\}\mid \bar{a}_l\in R_\sigma\}$. For each $j\in J$, choose some index $r_j\in \{1,\dots,m\}$, such that $a_{r_j}^j\in Z_\sigma$. Furthermore, for each $l\in L$ choose indices $t_l<s_l\in \{1,\dots,m\}$, such that $a_{t_l}^l=a_{s_l}^l$. Let $P:=\{r_j\mid j\in J\}\cup \{t_l\mid l\in L\}\cup \{s_l\mid l\in L\}$. Note that $|P|\leq 2k$. Hence, we can find indices $1\leq i_1<i_2<\cdots <i_n\leq m$ such that $P\subseteq \{i_1,\dots,i_n\}$. It follows that $(i_1,\dots,i_n)\in {\tau}_{1} (\bar{a}_1)\cap {\tau}_{1}
  (\bar{a}_2)\cap\dots \cap {\tau}_{1} (\bar{a}_k)\subseteq
  {\tau}_{1} (\bar{b})$, so $(b_{i_1},\dots,b_{i_n})\in \varrho$. Since $T_{\varrho}=\emptyset$, it follows that $(b_{i_1},\dots,b_{i_n})\in C_\varrho$. So for some $j\in\{1,\dots,n\}$ we have $b_{i_j}\in Z_\varrho=Z_\sigma$. But this implies $\bar{b}\in C_\sigma\subseteq \sigma$.

  $(\Rightarrow)$
  By Lemma~\ref{notail}, we obtain immediately that $T_\sigma=\emptyset$, $Z_{\varrho}=Z_{\sigma}$ and $n<m\leq |A|-1$, so it is left to show that $2k\leq n$. Suppose that $n<2k$. We will show that then there exist $\bar{a}_1,\dots ,\bar{a}_k\in \sigma$ such that ${\type}_{\varrho} (\bar{a}_1)\cap {\type}_{\varrho}
  (\bar{a}_2)\cap\dots \cap {\type}_{\varrho} (\bar{a}_k)=(\emptyset,\emptyset)$. If we succeed in this endeavor, then Proposition~\ref{Polcentralnih} implies $\sigma=A^m$, a contradiction.  
  
 It remains to construct $\bar{a}_1,\dots ,\bar{a}_k$. Let $\bar{b}\in A^m$ be such that $\type_{\varrho} (\bar{b})=(\emptyset,\emptyset)$. As $T_\varrho=\emptyset$ and $Z_\varrho=Z_\sigma$, any element from $A^m\setminus \sigma$ will do.
  
  The $m$-tuples $\bar{a}_1,\bar{a}_2,\dots ,\bar{a}_{\lfloor\frac{n}{2}\rfloor},\bar{a}_{\lfloor\frac{n}{2}\rfloor +1}$ are constructed using elements from $\{b_1,\dots ,b_m\}$ as entries. In case that $n$ is odd or $m> n+1$, we define
  \begin{equation}\tag{$\ast$}\label{defai}
  \bar{a}_i:=(\dots, \underset{\substack{\uparrow\\a_{2i-1}^i}}{b_i},\underset{\substack{\uparrow\\a_{2i}^i}}{b_i},\dots),\qquad i=1,2,\dots ,\bigg\lfloor\frac{n}{2}\bigg\rfloor +1,
  \end{equation} where all other entries are distinct and from the set $\{b_1,\dots,b_m\}\setminus \{b_i\}$.
  Otherwise, if $n$ is even and $m=n+1$, then, for all $i\in\{1,\dots,\frac{n}{2}\}$ we define $\bar{a}_i$ as above in \eqref{defai}. Moreover, we define
   \[
  \bar{a}_{\frac{n}{2}+1}:=(b_1,\dots,b_{m-1},c),
  \] where $c\in Z_{\sigma}(=Z_{\varrho})$.
  Note that since $n<2k$, it follows that $\lfloor\frac{n}{2}\rfloor +1\leq k$. If $\lfloor \frac{n}{2}\rfloor+1<k$ the remaining tuples $\bar{a}_i$, for $\lfloor \frac{n}{2}\rfloor+2\leq i\leq k$ we choose arbitrarily.
  
  Observe that then 
  \[
  \bigcap_{i=1}^{\lfloor\frac{n}{2}\rfloor +1} {\type}_{\varrho} (\bar{a}_i)\supseteq \bigcap_{i=1}^k {\type}_{\varrho} (\bar{a}_i).
  \]
  
  Let us compute $\bigcap_{i=1}^{\lfloor\frac{n}{2}\rfloor +1} {\type}_{\varrho} (\bar{a}_i)$.
  %Note that ${\tau}_2(\bar{a}_i)=\{(2i-1,2i)\}$.
  
  It is clear that 
  \[
  \bigcap_{i=1}^{\lfloor\frac{n}{2}\rfloor +1} {\tau}_{2} (\bar{a}_i)=\emptyset.
  \]
  We will show that the same holds for $\bigcap_{i=1}^{\lfloor\frac{n}{2}\rfloor +1} {\tau}_{1} (\bar{a}_i)$.
  
  Let us first treat the case when $n$ is odd or $m> n+1$.
  Suppose $(j_1,\dots,j_n)\in \bigcap_{i=1}^{\lfloor\frac{n}{2}\rfloor +1} {\tau}_{1} (\bar{a}_i)$. Then for each $i\in \{1,\dots,\lfloor\frac{n}{2}\rfloor+1\}$, since $(j_1,\dots,j_n)\in \tau_1(\bar{a}_i)$, we have that $\{2i-1,2i\}\subseteq \{j_1,\dots,j_n\}$. It follows that $|\{j_1,\dots,j_n\}|\geq 2\cdot (\lfloor \frac{n}{2}\rfloor+1)>n$, a contradiction.
  
  Hence, 
  \begin{equation}\tag{$\ast\ast$}\label{presek}
  \bigcap_{i=1}^{\lfloor\frac{n}{2}\rfloor +1} {\type}_{\varrho} (\bar{a}_i)=(\emptyset,\emptyset).
  \end{equation}
  
  In case that $n$ is even and $m=n+1$, we argue as follows:
  
  Note that every tuple from $\tau_1(\bar{a}_{\frac{n}{2}+1})$ contains $m$ as an entry. Suppose that $\bigcap_{i=1}^{\frac{n}{2} +1} {\tau}_{1} (\bar{a}_i)\neq \emptyset$. Then it contains a tuple $(i_1,\dots,i_{n-1},m)$, where $i_1<\cdots<i_{n-1}<m$. Since every tuple from $\tau_1(\bar{a}_i)$, $i=1,\dots,\frac{n}{2}$ has to contain entries $2i-1$ and $2i
  $, it follows that $\{1,\dots,n\}=\bigcup_{i=1}^{\frac{n}{2}}\{2i-1,2i\}\subseteq \{i_1,\dots,i_{n-1}\}$ --- a contradiction. Hence, \eqref{presek} holds in this case, too.
 
 Altogether we proved
  \[
  \bigcap_{i=1}^{k} {\type}_{\varrho} (\bar{a}_i)=(\emptyset,\emptyset).
  \]
  It follows that $\sigma=A^m$, which is a contradiction.
  \end{proof}
  
  \begin{COR}
    For each $k\geq 2$ the height of the poset of $k$-ary parts of maximal clones on a set $A$ with $|A| \ge 2k+1$ is at least $|A| - 2k$.
  \end{COR}
  \begin{proof}
    Fix a $c \in A$ and consider a sequence of central relations
    $\varrho_i$, $i \in \{0, \ldots, |A| - 2k - 1\}$ such that
    $\ar(\varrho_i) = 2k + i$, $Z_{\varrho_i} = \{c\}$ and $T_{\varrho_i}=\emptyset$. Then, by Theorem~\ref{main-thm} we have that
   \[
      \Pol_k \varrho_0 \subseteq \Pol_k \varrho_1 \subset \ldots \subset \Pol_k \varrho_{|A|-2k-1}.
   \]
    This concludes the proof.
  \end{proof}

  \section{Acknowledgements}
  
  We thank to anonymous referees for their helpful comments. Special acknowledgements go to one of the referees that pointed out a missing case in the proof of Theorem~\ref{main-thm}, and suggested a way to fill this gap.

  The authors gratefully acknowledge the financial support of the Ministry of Science, Technological Development and Innovation of the Republic of Serbia (Grant No. 451-03-47/2023-01/200125).

%\paragraph{Data Availability Statement:} Data sharing not applicable to this article as datasets were neither generated nor analysed.

\subsection*{Declarations}

\paragraph{Conflict of interest:} The authors declare that they have no relevant financial or non-financial interests to disclose.

\end{document}